\documentclass[a4paper,leqno]{amsart}

\usepackage{diagrams}

\usepackage{amssymb,amsthm,amsmath,amsrefs}

\usepackage{ifthen}

\newcounter{mylisti} \newcounter{mylistii}
\newcounter{nest}
\newcommand{\defaultlabel}{}

\newenvironment{mylist}[1]{%
  \addtocounter{nest}{1}
  \ifthenelse{\value{nest}=1}{%
    \renewcommand{\defaultlabel}{(\roman{mylisti})\hfill}}{%
    \renewcommand{\defaultlabel}{(\alph{mylistii})\hfill}}
  \begin{list}{\defaultlabel}{%
      \ifthenelse{\value{nest}=1}{\usecounter{mylisti}}{%
        \usecounter{mylistii}}
      
      \addtolength{\itemsep}{0.5ex}
      \settowidth{\labelwidth}{#1}
      \setlength{\leftmargin}{\labelwidth}
      \addtolength{\leftmargin}{\labelsep}}}{\addtocounter{nest}{-1}
\end{list}}

\newcommand{\be}{\ensuremath{\mathbb E}}

\newcommand{\bn}{\ensuremath{\mathbb N}}
\newcommand{\bp}{\ensuremath{\mathbb P}}

\newcommand{\br}{\ensuremath{\mathbb R}}

\newcommand{\cI}{\ensuremath{\mathcal I}}
\newcommand{\cJ}{\ensuremath{\mathcal J}}
\newcommand{\cK}{\ensuremath{\mathcal K}}
\newcommand{\cL}{\ensuremath{\mathcal L}}

\newcommand{\cP}{\ensuremath{\mathcal P}}

\newcommand{\cS}{\ensuremath{\mathcal S}}

\newcommand{\fB}{\ensuremath{\mathfrak{B}}}
\newcommand{\fH}{\ensuremath{\mathfrak{H}}}

\newcommand{\abs}[1]{\lvert #1\rvert}
\newcommand{\bigabs}[1]{\big\lvert #1\big\rvert}
\newcommand{\Bigabs}[1]{\Big\lvert #1\Big\rvert}

\newcommand{\conv}{\ensuremath{\mathrm{conv}}}

\newcommand{\diag}{\operatorname{diag}}

\newcommand{\fss}{\ensuremath{\mathcal F\mathcal S}}

\newcommand{\ip}[2]{\ensuremath{\langle #1,#2\rangle}}
\newcommand{\bigip}[2]{\ensuremath{\big\langle #1,#2\big\rangle}}

\newcommand{\norm}[1]{\lVert #1\rVert}
\newcommand{\bignorm}[1]{\big\lVert #1\big\rVert}
\newcommand{\Bignorm}[1]{\Big\lVert #1\Big\rVert}

\newcommand{\spn}{\ensuremath{\mathrm{span}}}
\newcommand{\cspn}{\ensuremath{\overline{\mathrm{span}}}}

\newcommand{\symdif}{\bigtriangleup}

\newcommand{\co}{\mathrm{c}_0}

\newcommand{\vare}{\varepsilon}
\newcommand{\varf}{\varphi}

\renewcommand{\geq}{\geqslant}
\renewcommand{\leq}{\leqslant}

\newcommand{\ds}{\displaystyle}

\newcommand{\eg}{\textit{e.g.,}\ }

\newcommand{\ie}{\textit{i.e.,}\ }

\newtheorem{thm}{Theorem}
\newtheorem{mainthm}{Theorem}

\newtheorem*{mainproblem*}{Problem}

\newtheorem{lem}[thm]{Lemma}
\newtheorem{prop}[thm]{Proposition}
\newtheorem{cor}[thm]{Corollary}
\newtheorem{problem}[thm]{Problem}

\theoremstyle{definition}

\theoremstyle{remark}

\newtheorem*{rem}{Remark}

\begin{document}

\title{Closed ideals of operators between the classical sequence
  spaces}

\author{D.~Freeman}
\address{Department of Mathematics and Statistics, St Louis
  University, St Louis, MO 63103  USA}
\email{dfreema7@slu.edu}

\author{Th.~Schlumprecht}
\address{Department of Mathematics, Texas A\&M University, College
  Station, TX 77843, USA and Faculty of Electrical Engineering, Czech
  Technical University in Prague,  Zikova 4, 166 27, Prague}
\email{schlump@math.tamu.edu}

\author{A.~Zs\'ak}
\address{Peterhouse, Cambridge, CB2 1RD, UK}
\email{a.zsak@dpmms.cam.ac.uk}

\date{23 September 2016}

\thanks{The first author was supported by grant 353293 from the
  Simon's Foundation. The second author's research was supported by
  NSF grant DMS1464713. The first and third authors were supported by
  the 2016 Workshop in  Analysis and Probability at Texas A\&M
  University.}
\subjclass[2010]{}

\begin{abstract}
  We prove that the spaces $\cL(\ell_p,\co)$,
  $\cL(\ell_p,\ell_\infty)$
  and $\cL(\ell_1,\ell_q)$ of operators with $1<p,q<\infty$ have
  continuum many closed ideals. This extends and improves earlier
  works by Schlumprecht and Zs\'ak~\cite{sz:15}, by
  Wallis~\cite{wallis:15} and by Sirotkin and
  Wallis~\cite{sirotkin-wallis:16}. Several open problems remain. Key
  to our construction of closed ideals are matrices with the
  Restricted Isometry Property that come from Compressed Sensing.
\end{abstract}

\maketitle

\section{Introduction}

Closed ideals of operators between Banach spaces is a very old
subject going back to at least the 1940s when Calkin proved that the
compact operators are the only non-trivial, proper closed ideal of
the algebra $\cL(\ell_2)$ of all bounded linear maps on Hilbert
space~\cite{calkin:41}. Two decades later Gohberg, Markus and
Fel{\cprime}dman~\cite{goh-mar-fel:60} showed that the same result
holds for the algebras $\cL(\ell_p)$, $1\leq p<\infty$, and
$\cL(\co)$. It is perhaps surprising that the situation is vastly more
complicated for the spaces $\ell_p\oplus \ell_q$ and $\ell_p\oplus
\co$, $1\leq p<q<\infty$. We refer the reader to Pietsch's book
`Operator Ideals'~\cite{pietsch:78}*{Chapter 5} and to sections~1
and~2.3 of~\cite{sz:15} for a detailed history of the study of
closed ideals on these spaces. Here we shall be fairly brief and
concentrate on developments that best place our new results in
context.

Let $X=\ell_p$ and let $Y$ be either $\ell_q$ or $\co$, where $1\leq
p<q<\infty$. We recall~\cite{volkmann:76} that $\cL(X\oplus Y)$ has
exactly two maximal ideals: the closures of the ideals of operators
factoring through $X$ and $Y$, respectively. Moreover, the set of all
non-maximal, proper closed ideals of $\cL(X\oplus Y)$ is in a
one-to-one, inclusion-preserving correspondence with the set of
all closed ideals of $\cL(X,Y)$ (see~\cite{pietsch:78}*{Theorem 5.3.2}
or Section~2 of~\cite{schlump:12}). Here an ideal of $\cL(X,Y)$ is a
subspace $\cJ$ of $\cL(X,Y)$ with the \emph{ideal property}:
$ATB\in\cJ$ whenever $A\in\cL(Y), T\in\cJ$ and $B\in\cL(X)$. In his
book, Pietsch raised the problem whether $\cL(X,Y)$ has infinitely many
closed ideals in the case $X=\ell_p$ and $Y=\ell_q$ with $1\leq
p<q<\infty$. This problem remained open for nearly forty years.
It is straightforward that the smallest non-trivial closed ideal is
the ideal of compact operators and that every other non-trivial closed
ideal must contain the closed ideal generated by the formal inclusion
map $I_{X,Y}\colon X\to Y$. However, it is not even obvious if there
are any other non-trivial, proper closed ideals besides the compact
operators. The first results in this direction are due to
Milman~\cite{milman:70} who first proved that $I_{X,Y}$ is
finitely strictly singular, and then exhibited in the case
$1<p<q<\infty$ an operator in $\cL(\ell_p,\ell_q)$ that is not
finitely strictly singular. (Definitions will be given in
Section~\ref{sec:main-results} below.) More recently, further closed
ideals, but still only finitely many, were found by Sari,
Tomczak-Jaegermann, Troitsky and
Schlumprecht~\cite{sari-schlump-tomczak-troitsky:07} and
by Schlumprecht~\cite{schlump:12}.

Finally, Schlumprecht and Zs\'ak found a positive answer to
Pietsch's question in the reflexive range $1<p<q<\infty$. In fact,
they showed~\cite{sz:15} that in that range, $\cL(\ell_p,\ell_q)$
contains continuum many closed ideals which with respect to inclusion
are order-isomorphic to $\br$. Shortly after this, Wallis
observed~\cite{wallis:15} that the techniques of~\cite{sz:15} extend
to prove the same result for $\cL(\ell_p,\co)$ in the range $1<p<2$,
and for $\cL(\ell_1,\ell_q)$ in the range $2<q<\infty$. The answer to
Pietsch's question was completed by Sirotkin and
Wallis~\cite{sirotkin-wallis:16} who proved that there are
uncountably many closed ideals in $\cL(\ell_1,\ell_q)$ for
$1<q\leq\infty$ as well as in $\cL(\ell_1,\co)$ and in
$\cL(\ell_p,\ell_\infty)$ for $1\leq p<\infty$. The techniques used
in~\cite{sz:15}
and~\cite{sirotkin-wallis:16} are completely different and we shall
have more to say about this in the final section of this paper. The
method of Sirotkin and Wallis produces chains of closed ideals of
size~$\omega_1$, the first uncountable ordinal. Moreover, their method
can not tackle (for good reasons as we shall see later) the question
whether $\cL(\ell_p,\co)$ has infinitely many closed ideals in the
range $2\leq p<\infty$. The main focus of this paper is \emph{this
}remaining
open problem regarding closed ideals of operators between classical
sequence spaces. Our main result is a solution of this problem as well
as an extension of the result of Wallis for $1<p<2$.

\begin{mainthm}
  \label{mainthm:l_p-c_0}
  For all $1<p<\infty$ there are continuum many closed ideals in
  $\cL(\ell_p,\co)$.
\end{mainthm}

Our techniques are somewhat similar to those used
in~\cite{sz:15}. However, instead of using independent sequences of
3-valued, symmetric random variables spanning Rosenthal spaces, here
we are going to rely on matrices with the Restricted Isometry Property
(RIP for short) to
generate ideals. Their appearance in this problem was somewhat of a
surprise to us. In fact, we shall prove considerably more by showing
that the distributive lattice of closed ideals of $\cL(\ell_p,\co)$
has a rich structure (see Theorem~\ref{thm:structure-of-ideals}
below). From Theorem~\ref{mainthm:l_p-c_0} it is a short step to
extend the results of Sirotkin and Wallis as follows.

\begin{mainthm}
  \label{mainthm:l_1-l_q}
  For all $1<p<\infty$ there are continuum many closed ideals in
  $\cL(\ell_p,\ell_\infty)$ and $\cL(\ell_1,\ell_p)$.
\end{mainthm}

As before, we will present more precise statements about the lattice
of closed ideals for these spaces. We do not know if $\cL(\ell_1,\co)$
and $\cL(\ell_1,\ell_\infty)$ have continuum many closed ideals. Thus,
in these cases the best known result is the aforementioned theorem of
Sirotkin and Wallis~\cite{sirotkin-wallis:16} that shows the existence
of an $\omega_1$-chain of closed ideals.

The paper is organized as follows. Section~\ref{sec:main-results}
begins with definitions, notations and the introduction of RIP
vectors. We then present the main results about ideals in
$\cL(\ell_p,\co)$. In Section~\ref{sec:secondary-results} we present a
proof of Theorem~\ref{mainthm:l_1-l_q}. In the final section we first
explain the method of Sirotkin and Wallis and how it differs from
ours. We then prove some further results of interest about the ideal
structure in $\cL(\ell_p,\co)$. These explain why the techniques
of~\cite{sirotkin-wallis:16} cannot work here. We conclude with
further remarks and open problems.

To conclude this introduction, we mention two pieces of notation used
throughout the paper. Firstly, we denote by $X\cong Y$ if the Banach
spaces $X$ and $Y$ are isometrically isomorphic, whereas $X\sim Y$
indicates that they are merely isomorphic. Secondly, the action of a
functional $f\in X^*$ on a vector $x\in X$ will be written as
$\ip{x}{f}$. Our convention is that the vector always appears on the
left and the functional on the right.

\section{Closed ideals in $\cL(\ell_p,\co)$}
\label{sec:main-results}

Given Banach spaces $X$ and $Y$, we denote by $\cL(X,Y)$ the space of
all operators from $X$ to $Y$. By an operator we shall always mean a
bounded linear map. When $X=Y$ we write $\cL(X)$ instead of
$\cL(X,X)$. When $X$ is formally a subset of $Y$, we write $I_{X,Y}$
for the formal inclusion map $X\to Y$ provided this is bounded. We
shall sometimes write $I$ instead of $I_{X,Y}$ when $X$ and $Y$ are
clear from the context. Assume now that
$X=\big( \bigoplus_{n=1}^\infty X_n\big)_{\ell_p}$ for some
$1\leq p<\infty$ and that either
$Y=\big( \bigoplus_{n=1}^\infty Y_n\big)_{\ell_q}$ for some
$p\leq q\leq\infty$ or
$Y=\big( \bigoplus_{n=1}^\infty Y_n\big)_{\co}$. Then, given a
uniformly bounded sequence of operators $T_n\colon X_n\to Y_n$,
$n\in\bn$, we write $T=\diag (T_n)_{n\in\bn}$ for the diagonal
operator $T\colon X\to Y$ defined by $(x_n)\mapsto (T_nx_n)$.

By an \emph{(operator) ideal of $\cL(X,Y)$ }we
mean a subspace $\cJ$ of $\cL(X,Y)$ such that $ATB\in\cJ$ for all
$A\in\cL(Y)$, $T\in\cJ$ and $B\in\cL(X)$. A \emph{closed (operator)
  ideal }is an ideal that is closed in the operator norm. Note that
the closure of an ideal is a closed ideal. We shall denote by
$\fH(X,Y)$ the set of all closed ideals of $\cL(X,Y)$. Note that
$\fH(X,Y)$ is a distributive lattice under inclusion: the \emph{meet
}and \emph{join }operations are given by
\[
\cI\wedge \cJ=\cI\cap\cJ\ ,\quad \cI\vee\cJ=\overline{\cI+\cJ}\qquad
\text{for }\cI, \cJ\in\fH(X,Y)\ .
\]
We next recall some standard operator ideals. We denote by
$\cK(X,Y)$, $\fss(X,Y)$, and $\cS(X,Y)$ the closed ideals of,
respectively, compact, finitely strictly singular, and strictly
singular operators from $X$ to $Y$. Recall that $T\in \cL(X,Y)$ is
\emph{finitely strictly singular }if for all $\vare>0$ there exists
$n\in\bn$ such that for every subspace $E\subset X$ of dimension at
least $n$, there is a vector $x\in E$ with
$\norm{Tx}<\vare\norm{x}$. We say $T$ is \emph{strictly singular }if
no restriction of $T$ to an infinite-dimensional subspace of $X$ is an
isomorphism. It is clear that
$\cK(X,Y)\subset\fss(X,Y)\subset\cS(X,Y)$. When $X=Y$ these ideals
become $\cK(X)\subset\fss(X)\subset\cS(X)$. We shall also sometimes
write $\cK$, $\fss$ and $\cS$ when $X$ and $Y$ are clear from the
context. Recall that if $X=\ell_p$,
$1\leq p<\infty$, and either $Y=\ell_q$, $p<q<\infty$, or $Y=\co$, then
$\cL(X,Y)=\cS(X,Y)$ and $\cL(Y,X)=\cK(Y,X)$.

We next introduce notation for closed ideals generated by a fixed
operator. Let $T\colon W\to Z$ be a bounded linear map. For any pair
$(X,Y)$ of Banach spaces we denote by $\cJ^T(X,Y)$ the closed ideal of
$\cL(X,Y)$ generated by $T$. Thus,
\[
\cJ^T(X,Y)=\cspn \{ ATB:\, A\in\cL(Z,Y), B\in\cL(X,W) \}
\]
is the closed linear span of operators factoring through $T$. As
usual, we write $\cJ^T(X)$ instead of $\cJ^T(X,Y)$ when $X=Y$. We
shall also write $\cJ^T$ instead of $\cJ^T(X,Y)$ or $\cJ^T(X)$ when
$X$ and $Y$ are clear from the context.

We now begin our study of closed ideals in $\cL(\ell_p,\co)$ where
$1\leq p<\infty$ is fixed for the rest of this section. As
mentioned in the Introduction, we know of the following ideals:
\[
\{ 0\} \subsetneq \cK \subsetneq \cJ^{I_{\ell_p,\co}} \subset \fss
\subsetneq \cS =\cL(\ell_p,\co)\ .
\]
To see that $\fss$ is a proper ideal, fix for each $n\in\bn$ an
embedding $T_n\colon \ell_p^n\to \ell_\infty^{M_n}$ for a sufficiently
large $M_n\in\bn$ satisfying, say,
$\norm{T_nx}\leq \norm{x}\leq 2\norm{T_nx}$ for all
$x\in\ell_p^n$. Then the diagonal operator $\diag(T_n)_{n\in\bn}$ from
$\ell_p\cong \big(\bigoplus _{n\in\bn}\ell_p^n\big)_{\ell_p}$ to
$\co\cong \big(\bigoplus _{n\in\bn}\ell_\infty^{M_n}\big)_{\co}$ is
clearly not finitely strictly singular. As noted in~\cite{wallis:15},
the results of~\cite{sz:15} extend without much difficulty to show
that in the range $1<p<2$, there is a chain of closed ideals
order-isomorphic to $\br$ between $\cJ^{I_{\ell_p,\co}}$ and
$\fss$. However, for $2\leq p<\infty$ it was not even known if there
are any other non-trivial, proper closed ideals besides $\cK$ and
$\fss$. We remedy this situation with our main result,
Theorem~\ref{thm:structure-of-ideals} below. The main ingredient will
be certain RIP vectors which we introduce next.

Fix $\delta\in (0,1)$ and positive integers $k<n\leq N$. We say that
the $n\times N$ matrix $A$ satisfies the \emph{Restricted Isometry
  Property (RIP) of order $k$ with error $\delta$ }if for all
$x\in \br^N$ with at most $k$ non-zero entries the following holds:
\[
(1-\delta) \norm{x}_{\ell_2^N} \leq \norm{Ax}_{\ell_2^n} \leq
(1+\delta) \norm{x}_{\ell_2^N}\ .
\]
Cand\`es and Tao~\cite{candes-tao:05} introduced this property and
established the important r\^ole it plays in Compressed
Sensing. Briefly, such matrices allow recovery of sparse data from few
measurements. Cand\`es and Tao also established the existence of
matrices with the RIP: roughly speaking, they proved that a random
$n\times N$ matrix
with Gaussian entries satisfies RIP with overwhelming probability. In
particular, and this is what we will use below, for any
$\delta\in(0,1)$ and $k\in\bn$, there exists a constant
$c=c(\delta,k)>0$ such that for all $n, N\in\bn$ satisfying
$\frac{\log N}{n}<c$ there exists an $n\times N$ matrix with the RIP
of order $k$ and error $\delta$. Before
continuing, let us mention that the RIP phenomenon has also been
fundamental in asymptotic finite-dimensional Banach space
theory. Indeed, all the tools necessary for constructing RIP matrices
were already available in the 1970s and 1980s, for example, in
Kashin's decomposition of $\ell_1^{2n}$ into two $n$-dimensional
Euclidean sections~\cite{kashin:78}, and in the well known
Johnson--Lindenstrauss lemma concerning certain Lipschitz mappings of
a finite set in Euclidean space onto a space of dimension logarithmic
in the size of the set~\cite{johnson-lindenstrauss:84}. We refer the
reader to the article of Baraniuk, Davenport, DeVore and
Waking~\cite{bddw:08} where they discuss these connections and provide
a simple proof of the RIP for random matrices whose entries satisfy
certain concentration inequalities (this includes matrices with
independent Gaussian or Rademacher entries as well as others). The
book of Foucart and Rauhut~\cite{foucart-rauhut:13} also discusses
these topics in detail as well as many other aspects of the
mathematics of Compressed Sensing.

The following statement is a straightforward consequence of the above
discussion. There exists a sequence $u_1<v_1<u_2<v_2<\dots$ of
positive integers such that setting
\begin{equation}
  \label{eq:defn-of-s_n}
  s_n=\begin{cases}%
  (2u_n)^{p/2}\cdot n^p & \text{if }2\leq p<\infty\\
  2u_n\cdot n^2 & \text{if }1\leq p\leq 2
  \end{cases}
\end{equation}
the following properties hold:
\begin{equation}
  \label{eq:choice-of-u_n}
  u_n\geq 19 n^3\cdot (6n+1)^{u_1+u_2+\dots + u_{n-1}}\ ,
\end{equation}
\begin{equation}
  \label{eq:choice-of-v_n}
  v_n\geq 9 n^3 \cdot s_n\ ,
\end{equation}
and there exist unit vectors $\big( g^{(n)}_i\big)_{i=1}^{v_n}$ in
$\ell_2^{u_n}$ (the normalized columns of an RIP matrix) such that for
every $J\subset \{1,2,\dots,v_n\}$ with $\abs{J}\leq (s_{n-1}+1)\vee
19n^2$ we have
\begin{equation}
  \label{eq:almost-on}
  \frac12 \sum_{i\in J}\abs{a_i}^2 \leq \Bignorm{\sum_{i\in J} a_i
    g^{(n)}_i}^2 \leq 2\sum_{i\in J} \abs{a_i}^2\qquad \text{for all
  }(a_i)_{i\in J}\subset \br\ ,
\end{equation}
and
\begin{equation}
  \label{eq:besselian}
  \sum _{i\in J} \abs{\ip{x}{g^{(n)}_i}}^2 \leq 2\norm{x}^2 \qquad
  \text{for all }x\in\ell_2^{u_n}\ .
\end{equation}
We next define, for each $n\in\bn$, operators $T_n\colon
\ell_2^{u_n}\to\ell_\infty^{v_n}$ by $x\mapsto\big( \ip{x}{g^{(n)}_i}
\big)_{i=1}^{v_n}$. Note that $\norm{T_n}=1$, and hence we may use
these maps to define certain diagonal operators. Set
$U=\big(\bigoplus_{n=1}^\infty \ell_2^{u_n} \big)_{\ell_p}$ and
$V=\big(\bigoplus _{n=1}^\infty \ell_\infty^{v_n} \big)_{\co}$. For an
infinite set  $M\subset\bn$ define $T_M\colon U\to V$ by $(x_n)\mapsto
(y_n)$, where $y_n=T_n(x_n)$ if $n\in M$ and $y_n=0$ otherwise. The
following is the main result of this section.
\begin{thm}
  \label{thm:l-p-to-co}
  Let $M,N$ be infinite subsets of $\bn$. 
  \begin{mylist}{(ii)}
  \item
    If $M\setminus N$ is infinite, then $T_M\notin \cJ^{T_N}(U,V)$.
  \item
    If $N\setminus M$ is finite, then $\cJ^{T_N}(U,V)\subset
    \cJ^{T_M}(U,V)$.
  \end{mylist}
\end{thm}
As a corollary we will deduce that the lattice of closed ideals of
$\cL(U,V)$ has a rich structure in the following precise sense. Define
$\fB$ to be the Boolean algebra obtained as the quotient of the power
set $\cP\bn$ of $\bn$ by the equivalence relation $\sim$ defined by
$M\sim N$ if and only if $M\symdif N$ is finite. By
Theorem~\ref{thm:l-p-to-co}(ii) above, the ideal $\cJ^{T_M}(U,V)$
depends only on the equivalence class $[M]$ of the infinite set
$M$. Thus, we have a well-defined map
\begin{eqnarray*}
  \varf\colon\fB&\to&\fH(U,V)\\[1ex]
  [M] &\mapsto& %
  \begin{cases}
    \cJ^{T_M} & \text{if $M$ is infinite},\\
    \cJ^{I_{U,V}} & \text{if $M$ is finite.}
  \end{cases}
\end{eqnarray*}
Here $I_{U,V}$ is the formal inclusion of $U$ into $V$. Note that $U$
can indeed be viewed as a subset of $V$ in the following way. Every
element of $U=\big(\bigoplus_{n=1}^\infty \ell_2^{u_n} \big)_{\ell_p}$
can be thought of as a sequence $(x_i)_{i=1}^\infty$ with
$\sum_{n=1}^\infty\Big(\sum_{i=t_{n-1}+1}^{t_n}\abs{x_i}^2\Big)^{p/2}<\infty$,
where $t_n=u_1+\dots+u_n$ for all $n\in\bn$. On the other hand, $V$ is
nothing else but $\co$ with the particular choice
$\big(\bigoplus_{n=1}^\infty\ell_\infty^{v_n}\big)_{\co}$ of blocking
of its unit vector basis. It will also be helpful to think of
$I_{U,V}$ as a diagonal operator: write $V=\co$ as
$\big(\bigoplus_{n=1}^\infty\ell_\infty^{u_n}\big)_{\co}$ using a
different blocking of the unit vector basis. Then
$I_{U,V}\colon U\to V$ is the diagonal operator whose $n^{\text{th}}$
diagonal entry is the formal identity
$\ell_2^{u_n}\to\ell_\infty^{u_n}$. At the very end of this section we
will comment on the choice of $\varf([M])$ for $M$ finite.
\begin{thm}
  \label{thm:structure-of-ideals}
  The map $\varf$ defined above is injective, monotone and preserves
  the join operation. In particular, $\cL(U,V)$ has continuum many
  closed ideals between $\cJ^{I_{U,V}}$ and $\fss$.
\end{thm}
Since $V=\co$ and for $1<p<\infty$ the space $U$ is isomorphic to
$\ell_p$ by Pe{\l}czy\'nski's Decomposition Theorem (see \eg
\cite{czech-book}*{Theorem 6.24}), our first main result,
Theorem~\ref{mainthm:l_p-c_0}, stated in the Introduction follows
immediately.

We begin with the proof of Theorem~\ref{thm:l-p-to-co}. Of
course, part~(ii) is trivial since if $N\setminus M$ is finite, then a
finite-rank perturbation of $T_N$ factors through $T_M$. Let us now
explain why part~(i) has a chance of being true. An obvious way
for~(i) to fail would be the existence of an injection $m\mapsto n_m$
from $M$ into $N$ such that $T_m$ factors through $T_{n_m}$ for all
$m\in M$. Indeed, then $T_M$ would factor through $T_N$ and we would
immediately obtain $T_M\in \cJ^{T_N}$. However, for $m\in M\setminus
N$ and for any $n\in N$, it is hard for $T_m$ to factor through
$T_n$. Indeed, $T_m$ preserves the norm of the RIP vectors
$g^{(m)}_i$, $1\leq i\leq v_m$. When $m>n$, the number $v_m$ of these
vectors is massive compared to the dimension $u_n$ of the domain of
$T_n$, and an easy pigeonhole argument will show that one cannot have
a map $\ell_2^{u_m}\to \ell_2^{u_n}$  preserving the norm of so many
RIP vectors. The case $m<n$ is slightly more complicated. In this case,
$T_m$ factoring through $T_n$ says something about the action of $T_n$
on a subspace of small dimension (small compared to the dimension
$u_n$ of the domain of $T_n$). It will then follow
from~\eqref{eq:almost-on} and~\eqref{eq:besselian} that the
restriction of $T_n$ to this small-dimensional subspace, and hence
$T_m$, factors through the formal identity $I\colon
\ell_2^r\to\ell_\infty^r$, where $r\leq s_m$, and in particular $r$ is
still much smaller than the number $v_m$ of RIP vectors whose norm
$T_m$ preserves. Another simple pigeonhole argument will again show
that the formal identity $I$ cannot preserve the norm of so many RIP
vectors, but this time it will be crucial that $I$ is mapping into
an $\ell_\infty$-space.

We begin the proof of our main theorem with a beefed up version of the
claim above that the restriction of $T_n$ to a small-dimensional
subspace factors through the formal identity
$\ell_2^r\to\ell_\infty^r$ with a suitable bound on~$r$.

\begin{lem}
  \label{lem:factoring-through-formal-id}
  Fix $m\in\bn$, a subset $N$ of $\{n\in\bn:\,n>m\}$ and an operator
  \[
  B\colon\ell_2^{u_m}\to \Big(\bigoplus _{n\in N}
  \ell_2^{u_n}\Big)_{\ell_p}
  \]
  with $\norm{B}\leq 1$. Let $D\colon \big(\bigoplus
  _{n\in N} \ell_2^{u_n}\big)_{\ell_p} \to \big(\bigoplus _{n\in N}
  \ell_\infty^{v_n}\big)_{\co}$ be the diagonal operator
  $\diag(T_n)_{n\in N}$. Then there exist subsets
  $J_n\subset\{1,2,\dots, v_n\}$, $n\in N$, such that $\sum_{n\in N}
  \abs{J_n}\leq s_m$ and there is an approximate factorization
  \[
  \begin{diagram}[midshaft,l>=3em]
    \ell_2^{u_m} & \rTo^B & \Big(\bigoplus_{n\in N}
    \ell_2^{u_n}\Big)_{\ell_p} & \rTo^D & \Big(\bigoplus _{n\in N}
    \ell_\infty^{v_n}\Big)_{\co}\\
  \dTo^P &&&& \uTo_R \\
  \Big( \bigoplus_{n\in N} \ell_2^{J_n} \Big)_{\ell_p} && \rTo^I &&
  \Big( \bigoplus_{n\in N} \ell_\infty^{J_n} \Big)_{\co}
  \end{diagram}
  \]
  with $\norm{DB-RIP}\leq \frac{1}{m}$, where $I$ is the formal
  inclusion, $\norm{P}\leq 2$ and $\norm{R}\leq 1$.
\end{lem}

\begin{proof}
  Let us first observe that we need only consider the case $2\leq
  p<\infty$. This is because the case $1\leq p\leq 2$ reduces to the
  case $p=2$. Indeed, if $1\leq p\leq 2$, then $D\colon
  \big(\bigoplus_{n\in N} \ell_2^{u_n}\big)_{\ell_p} \to
  \big(\bigoplus _{n\in N}\ell_\infty^{v_n}\big)_{\co}$ factors
  through $D$ viewed as a map from $\big(\bigoplus_{n\in N}
  \ell_2^{u_n}\big)_{\ell_2} \to \big(\bigoplus _{n\in N}
  \ell_\infty^{v_n}\big)_{\co}$ via the formal identity
  $\big(\bigoplus_{n\in N} \ell_2^{u_n}\big)_{\ell_p}\to
  \big(\bigoplus_{n\in N} \ell_2^{u_n}\big)_{\ell_2}$. The case
  $p=2$ then provides the required approximate factorization via the
  formal identity $I\colon\ell_2^r\to \ell_\infty^r$ for some $r\leq
  s_m$. This completes the proof by taking $J_n=\{1,\dots,r\}\subset
  \{1,\dots, v_n\}$ for $n=\min N$, and $J_{n'}=\emptyset$ for all
  $n'\in N\setminus\{n\}$.

  Let us now assume that $2\leq p<\infty$, and we may of course also
  take $N\neq\emptyset$. Set $t=s_m+1$ and let $H$ be a subset of $\{
  (n,j):\, n\in N,\ 1\leq j\leq v_n\}$ of size $t$ such that
  \[
  \bignorm{B^* g^{(n)}_j} \geq \bignorm{B^* g^{(n')}_{j'}}\qquad
  \text{for all }(n,j)\in H,\ (n',j')\notin H\ .
  \]
  Set $J=\big\{ (n,j)\in H:\,
  \bignorm{B^*g^{(n)}_j}\geq\frac1{m}\big\}$. For each $n\in N$ letting
  $H_n=\{ j:\,(n,j)\in H\}$ and $J_n=\{ j:\,(n,j)\in J\}$, we obtain
  the following inequalities.
  \begin{align*}
    \tfrac1{m^2}\abs{J_n} &\leq \sum_{j\in H_n} \bignorm{B^*
      g^{(n)}_j}^2 = 
    \sum_{j\in H_n} \sum _{i=1}^{u_m} \bigabs{\bigip{B^*
        g^{(n)}_j}{e_i}}^2 \\[2ex]
    & = \sum_{i=1}^{u_m} \sum_{j\in H_n}
    \bigabs{\bigip{g^{(n)}_j}{Be_i}}^2 \leq 2\cdot \sum_{i=1}^{u_m}
    \norm{P_nBe_i}^2 \\[2ex]
    & \leq 2\cdot u_m^{\frac{p-2}{p}} \left( \sum_{i=1}^{u_m}
    \norm{P_nBe_i}^p \right)^{2/p}
  \end{align*}
  where $(e_i)_{i=1}^{u_m}$ stands for the unit vector basis of
  $\ell_2^{u_m}$ and $P_n$ denotes the canonical projection of
  $\big(\bigoplus_{n'\in N}\ell_2^{u_{n'}}\big)_{\ell_p}$ onto
  $\ell_2^{u_n}$. The inequality in the second line follows
  from~\eqref{eq:besselian} and from $\abs{H_n}\leq\abs{H}=t=
  s_m+1\leq s_{n-1}+1$. Finally, the last inequality uses H\"older's
  inequality. Raising the above to the power $p/2$ yields
  \[
  \abs{J_n}\leq \abs{J_n}^{p/2} \leq 2^{p/2}\cdot
  u_m^{\frac{p-2}{2}}\cdot m^p \cdot\sum_{i=1}^{u_m}
  \norm{P_nBe_i}^p\ .
  \]
  We next sum over $n\in N$ to obtain
  \[
  \sum_{n\in N}\abs{J_n} \leq 2^{p/2}\cdot u_m^{\frac{p-2}{2}}\cdot
  m^p \cdot\sum_{i=1}^{u_m} \norm{Be_i}^p\leq 2^{p/2}\cdot
  u_m^{p/2}\cdot m^p =s_m\ ,
  \]
  using the assumption that $\norm{B}\leq 1$. In particular, we have
  $\abs{J}=\sum_{n\in N}\abs{J_n}\leq s_m<\abs{H}$. It follows from
  the choice of $H$ that $\bignorm{B^*g^{(n)}_j}<\frac1{m}$ for all
  $(n,j)\notin J$. From this we immediately obtain the claimed
  approximate factorization of $DB$ by setting
  $P(x)=\big(\bigip{Bx}{g^{(n)}_j}_{j\in J_n} \big)_{n\in N}$ for
  $x\in \ell_2^{u_m}$, and setting $R$ equal the diagonal operator
  whose $n^{\text{th}}$ diagonal entry for $n\in N$ is the canonical
  embedding
  $\ell_\infty^{J_n}\to \ell_\infty ^{v_n}$ sending $(x_j)_{j\in J_n}$
  to $(y_j)_{j=1}^{v_n}$, where $y_j=x_j$ for $j\in J_n$ and $y_j=0$
  otherwise. Note that $\norm{P}\leq 2$ follows
  from~\eqref{eq:besselian} and from $\abs{J_n}\leq\abs{J}\leq s_m
  <s_{n-1}+1$.
\end{proof}

Before moving on to the proof of Theorem~\ref{thm:l-p-to-co}, we
observe the following consequence of
Lemma~\ref{lem:factoring-through-formal-id} that will be needed in the
proof of Theorem~\ref{thm:structure-of-ideals}.
\begin{cor}
  \label{cor:our-maps-are-fss}
  For each infinite set $M\subset\bn$, the operator $T_M\colon U\to V$
  is finitely strictly singular.
\end{cor}
\begin{proof}
  Fix an infinite set $M\subset\bn$ and $\vare>0$. We need to find
  $d\in\bn$ such that every subspace $E\subset U$ of dimension at
  least $d$ contains a vector $x$ with
  $\norm{T_Mx}<\vare\norm{x}$. Set $q=\max\{2,p\}$ and choose
  $m\in\bn$ with $m^{-1}+2m^{-1/q}<\vare/2$. Next choose a very large
  $d\in\bn$ so that every Banach space of dimension at least
  $d-(u_1+\dots+u_m)$ has a subspace $2$-isomorphic to
  $\ell_2^{u_m}$. Such a $d$ exists by Dvoretzky's theorem and depends
  only on $m$. We will show that this $d$ works.

  Let $E$ be an arbitrary $d$-dimensional subspace of $U$. Denote by
  $J$ the inclusion map from $E$ into $U$ and by $Q$ the canonical
  projection of $U$ onto
  $\big(\bigoplus_{i=1}^m\ell_2^{u_i}\big)_{\ell_p}$. Since the kernel
  of $QJ$ has dimension at least $d-(u_1+\dots+u_m)$, it contains a
  subspace $F$ that is $2$-isomorphic to $\ell_2^{u_m}$. Let $B$ be
  the restriction of $J$ to $F$. We need to show that for some $x\in
  F$ we have $\norm{T_MBx}<\vare\norm{x}$. It is clearly sufficient to
  show that if $N=\{n\in M:\,n>m\}$ and $B$ is \emph{any }operator
  from $\ell_2^{u_m}$ to
  $\big(\bigoplus_{n\in N} \ell_2^{u_n}\big)_{\ell_p}$ with
  $\norm{B}\leq 1$, then for some $x\in \ell_2^{u_m}$ we have
  $\norm{DBx}<\frac{\vare}{2} \norm{x}$, where
  $D=\diag(T_n)_{n\in N}$. At this point we can apply
  Lemma~\ref{lem:factoring-through-formal-id} to obtain the
  approximately commuting diagram
  \[
  \begin{diagram}[midshaft,l>=3em]
    \ell_2^{u_m} & \rTo^B & \Big(\bigoplus_{n\in N}
    \ell_2^{u_n}\Big)_{\ell_p} & \rTo^D & \Big(\bigoplus _{n\in N}
    \ell_\infty^{v_n}\Big)_{\co}\\
  \dTo^P &&&& \uTo_R \\
  \Big( \bigoplus_{n\in N} \ell_2^{J_n} \Big)_{\ell_p} && \rTo^I &&
  \Big( \bigoplus_{n\in N} \ell_\infty^{J_n} \Big)_{\co}
  \end{diagram}
  \]
  with $\norm{DB-RIP}\leq \frac{1}{m}$, where $I$ is the formal
  inclusion, $\norm{P}\leq 2$ and $\norm{R}\leq 1$.

  Now, if $x$ is a non-zero vector in $\ker P$, then $\norm{DBx}\leq
  \frac1{m}\norm{x}<\frac{\vare}2 \norm{x}$, and we are done. So we
  may assume that the image of $P$ has dimension at least $u_m$, and
  so at least $m$. A result of Milman~\cite{milman:70} (see
  also~\cite{sari-schlump-tomczak-troitsky:07}*{Lemma 3.4}) states
  that every $m$-dimensional subspace of $\co$ contains a non-zero
  vector which has at least $m$ co-ordinates of largest
  magnitude. Thus, there exists such a non-zero vector $y$ in the
  image of $P$. It follows that $\norm{y}\geq \norm{y}_{\ell_q}\geq
  m^{1/q} \norm{Iy}$. Choosing $x\in\ell_2^{u_m}$ with $Px=y$ we have
  \[
  \norm{DBx}\leq \frac1{m}\norm{x}+\norm{RIPx}\leq
  \frac1{m}\norm{x}+2m^{-1/q}\norm{x}<\frac{\vare}2 \norm{x}\ ,
  \]
  as required.
\end{proof}

\begin{proof}[Proof of Theorem~\ref{thm:l-p-to-co}]
  As mentioned earlier, we only need to prove~(i) as~(ii) is
  trivial. We shall construct a functional $\Phi$ in $\cL(U,V)^*$ that
  separates $T_M$ from $\cJ^{T_N}(U,V)$.
  
  For each $m\in\bn$ and $S\in\cL(U,V)$ define
  \[
  \Phi_m(S)=\tfrac{1}{v_m} \sum_{i=1}^{v_m}
  \bigip{Sg^{(m)}_i}{e^{(m)}_i}
  \]
  where $(e^{(m)}_i)_{i=1}^{v_m}$ denotes the unit vector basis of
  $\ell_1^{v_m}$ viewed in the obvious way as elements of
  $V^*\cong \big(\bigoplus_{n=1}^\infty
  \ell_1^{v_n}\big)_{\ell_1}$. Clearly, $\Phi_m\in\cL(U,V)^*$ with
  $\norm{\Phi_m}\leq 1$ for all $m\in\bn$, and $\Phi_m(T_M)=1$ for all
  $m\in M$. We are going to show that
  \begin{equation}
    \label{eq:kernel-of-phi}
    \lim_{m\to\infty,\ m\in M\setminus N} \Phi_m(AT_NB)=0 \qquad
    \text{for all }A\in\cL(V),\ B\in\cL(U)\ .
  \end{equation}
  It then follows that if $\Phi$ is a weak$^*$-accumulation point of
  $(\Phi_m)_{m\in M\setminus N}$ in $\cL(U,V)^*$, then $\Phi(T_M)=1$
  and $\cJ^{T_N}(U,V)\subset\ker \Phi$. The proof of the theorem is
  then complete.

  To see~\eqref{eq:kernel-of-phi}, let us fix $m\in M\setminus N$ and
  operators $A\in\cL(V)$ and $B\in\cL(U)$. We shall as we may assume
  that $\norm{A}\leq 1$ and $\norm{B}\leq 1$. We first observe that
  \begin{equation}
    \label{eq:ub-on-phi-m}
    \bigabs{\Phi_m(AT_NB)} = \Bigabs{\tfrac{1}{v_m} \sum_{i=1}^{v_m}
      \bigip{T_NBg^{(m)}_i}{A^*e^{(m)}_i}} \leq \tfrac1{v_m}
    \sum_{i=1}^{v_m} \bignorm{T_NBg^{(m)}_i}\ .
  \end{equation}
  We then split the right-hand side of~\eqref{eq:ub-on-phi-m} into two
  sums as follows. Let
  $n_0=\min \{n\in N:\, n>m\}$. Let $B^{(1)}$ be the restriction of
  $B$ to $\ell_2^{u_m}$ followed by the canonical projection of $U$
  onto $\big( \bigoplus _{n\in N, n < n_0} \ell_2^{u_n}
  \big)_{\ell_p}$. Similarly, let $B^{(2)}$ be the restriction of $B$
  to $\ell_2^{u_m}$ followed by the canonical projection of $U$ onto
  $\big( \bigoplus _{n\in N, n \geq n_0} \ell_2^{u_n}
  \big)_{\ell_p}$. We also let $D^{(1)}=\diag (T_n)_{n\in N, n<n_0}$
  and $D^{(2)}=\diag (T_n)_{n\in N, n\geq n_0}$.
  Continuing~\eqref{eq:ub-on-phi-m}, we next obtain
  \begin{equation}
    \label{eq:ub-on-phi-m-split}
      \bigabs{\Phi_m(AT_NB)} \leq  \tfrac1{v_m} \sum_{i=1}^{v_m}
      \bignorm{D^{(1)}B^{(1)}g^{(m)}_i}+\tfrac1{v_m} \sum_{i=1}^{v_m}
      \bignorm{D^{(2)}B^{(2)}g^{(m)}_i}\ .
  \end{equation}
  We shall now estimate the two terms above separately. Beginning with
  the first term, let us fix a $\frac1{3m}$-net $F$ in the unit ball
  of $\big( \bigoplus _{n\in N, n<n_0} \ell_2^{u_n}
  \big)_{\ell_p}$. Note that the dimension of this space is
  $\sum_{n\in N, n<n_0} u_n\leq u_1+u_2+\dots + u_{m-1}$. By standard
  volume estimate, we can find such an $F$ with $\abs{F}\leq
  (6m+1)^{u_1+\dots+u_{m-1}}$ by taking it to be a maximal
  $\frac1{3m}$-separated subset of the ball. Now, set $H=\big\{
  i:\,\bignorm{B^{(1)}g^{(m)}_i}>\frac1{m} \big\}$, and assume for a
  contradiction that $\abs{H}>\frac{v_m}{m}$. Then by the pigeon-hole
  principle we find $x\in F$ and $J\subset H$ such that $\abs{J}\geq
  \frac{v_m}{m\abs{F}}$ and $\bignorm{B^{(1)}g^{(m)}_i -x}\leq
  \frac1{3m}$ for all $i\in J$. It follows
  from~\eqref{eq:choice-of-u_n} that $\abs{J}\geq 19m^2$, and
  after replacing $J$ with a smaller set, we may in fact assume that
  $\abs{J}=19m^2$. It then follows by~\eqref{eq:almost-on} that
  \[
  \Bignorm{\sum_{i\in J} g^{(m)}_i}^2\leq 2\abs{J}\ .
  \]
  On the other hand, the choice of $J$ yields
  \[
  \Bignorm{\sum_{i\in J} B^{(1)}g^{(m)}_i}\geq \abs{J}\cdot
  \frac{2}{3m}-\abs{J}\cdot \frac{1}{3m}=\frac{\abs{J}}{3m}\ .
  \]
  The last two inequalities and the fact that
  $\norm{B^{(1)}}\leq\norm{B}\leq 1$ imply that $\abs{J}\leq 18m^2$
  which 
  is a contradiction. This shows that $\abs{H}\leq \frac{v_m}{m}$, and
  hence
  \begin{equation}
    \label{eq:ub-on-phi-m-1}
    \tfrac1{v_m} \sum_{i=1}^{v_m}\bignorm{D^{(1)}B^{(1)}g^{(m)}_i}
    \leq \frac{\abs{H}}{v_m}+\frac1{m}\leq \frac2{m}\ .
  \end{equation}
  To estimate the second term on the right-hand side
  of~\eqref{eq:ub-on-phi-m-split}, we shall first apply
  Lemma~\ref{lem:factoring-through-formal-id} with $B, D, N$ replaced
  by $B^{(2)}, D^{(2)}$ and $\{n\in N:\, n\geq n_0\}$, respectively, to
  obtain $J_n\subset\{1,\dots,v_n\}$, $n\in N, n\geq n_0$, with
  $\sum\abs{J_n}\leq s_m$, and an almost commuting diagram
  \[
  \begin{diagram}[midshaft,l>=3em]
    \ell_2^{u_m} & \rTo^{B^{(2)}} & \Big(\bigoplus_{n\in N, n\geq n_0}
    \ell_2^{u_n}\Big)_{\ell_p} & \rTo^{D^{(2)}} & \Big(\bigoplus
    _{n\in N, n\geq n_0} \ell_\infty^{v_n}\Big)_{\co}\\
  \dTo^P &&&& \uTo_R \\
  \Big( \bigoplus_{n\in N, n\geq n_0} \ell_2^{J_n} \Big)_{\ell_p} &&
  \rTo^I && \Big( \bigoplus_{n\in N, n\geq n_0} \ell_\infty^{J_n}
  \Big)_{\co}
  \end{diagram}
  \]
  with $\norm{D^{(2)}B^{(2)}-RIP}\leq \frac1{m}$, $\norm{P}\leq 2$ and
  $\norm{R}\leq 1$. From this we obtain
  \begin{equation}
    \label{eq:ub-on-phi-m-2}
    \tfrac1{v_m} \sum_{i=1}^{v_m}\bignorm{D^{(2)}B^{(2)}g^{(m)}_i}
    \leq \tfrac1{v_m} \sum_{i=1}^{v_m}\bignorm{IP g^{(m)}_i}
    +\frac1{m}\ .
  \end{equation}
  Let us set $H=\big\{i:\,\bignorm{IP g^{(m)}_i}>\frac1{m}\big\}$, and
  assume for a contradiction that $\abs{H}>\frac{v_m}{m}$. For each
  $i\in H$, there exist $n\in N, n\geq n_0$, and $j\in J_n$ with
  $\bigabs{\big[ Pg^{(m)}_i\big]_{n,j}}\geq \frac1{m}$, where
  $[y]_{n,j}$ denotes the $(n,j)$-coordinate of an element $y$
  of $\big( \bigoplus_{n\in N, n\geq n_0}\ell_2^{J_n}
  \big)_{\ell_p}$. Hence by pigeon-hole principle, there exist $n\in
  N, n\geq n_0$, $j\in J_n$ and $J\subset H$ with $\abs{J}\geq
  \frac{\abs{H}}{\sum\abs{J_n}}$ such that $\bigabs{\big[
      Pg^{(m)}_i\big]_{n,j}}\geq \frac1{m}$ for all $i\in J$. It
  follows from~\eqref{eq:choice-of-v_n} that
  $\abs{J}>\frac{v_m}{m\cdot s_m}\geq 9m^2$. After replacing $J$ by a
  smaller set, we may in fact assume that $\abs{J}=9m^2$. Hence
  by~\eqref{eq:almost-on} we have
  \[
  \Bignorm{\sum_{i\in J}\vare_ig^{(m)}_i}^2\leq 2\abs{J}
  \]
  where $\vare_i$ is the sign of $\big[ Pg^{(m)}_i\big]_{n,j}$ for
  each $i\in J$. On the other hand, by the choice of $J$ and since
  $\norm{P}\leq 2$, we get
  \[
  2 \Bignorm{\sum_{i\in J}\vare_ig^{(m)}_i}\geq
  \Bignorm{\sum_{i\in J}\vare_iPg^{(m)}_i}\geq 
  \sum_{i\in J}\vare_i \big[P g^{(m)}_i\big]_{n,j} \geq
  \abs{J}\cdot
  \frac1{m}\ .
  \]
  The last two inequalities yield $\abs{J}\leq 8m^2$ which is a
  contradiction. Hence $\abs{H}\leq \frac{v_m}{m}$ and
  \[
  \tfrac1{v_m} \sum_{i=1}^{v_m}\bignorm{IP g^{(m)}_i}\leq
  \frac{2\abs{H}}{v_m}+\frac1{m} \leq \frac3{m}\ .
  \]
  This inequality together with~\eqref{eq:ub-on-phi-m-2} yields the
  upper bound
  \[
  \tfrac1{v_m} \sum_{i=1}^{v_m}\bignorm{D^{(2)}B^{(2)}g^{(m)}_i}\leq
  \frac4{m}\ .
  \]
  Combining this with~\eqref{eq:ub-on-phi-m-1}
  and~\eqref{eq:ub-on-phi-m-split}, we finally obtain
  \[
  \bigabs{\Phi_m(AT_NB)} \leq \frac{6}{m}
  \]
  for all $m\in M\setminus N$. This completes the proof
  of~\eqref{eq:kernel-of-phi}.
\end{proof}

We conclude this section with a proof of
Theorem~\ref{thm:structure-of-ideals} which begins with a lemma.

\begin{lem}
  \label{lem:formal-id-factors-through-T_n}
  For each $m\in\bn$, the formal identity
  $I\colon\ell_2^m\to\ell_\infty^m$ factors through $T_n$ for all
  sufficiently large $n\in\bn$ via operators that are uniformly
  bounded.
\end{lem}
\begin{proof}
  Let us fix $m\in\bn$. Assume that $n\in\bn$ satisfies
  \[
  m\cdot \sqrt{\frac{m(m-1)}{M-1}} <\frac12\ ,\quad\text{where
  }M=(s_{n-1}+1)\vee 19n^2\ .
  \]
  Then by~\eqref{eq:besselian} we have
  \[
  \sum_{i,j=1}^M \bigabs{\bigip{g^{(n)}_i}{g^{(n)}_j}}^2 \leq 2M\ ,
  \]
  and hence
  \[
  \sum_{%
    \begin{subarray}{c}
      i,j =1\\[1pt]
      i\neq j
    \end{subarray}
  }^M \bigabs{\bigip{g^{(n)}_i}{g^{(n)}_j}}^2 \leq M\ .
  \]
  Next, let $S$ be a random (with respect to the uniform distribution)
  subset of $\{1,2,\dots,M\}$ of size
  $m$. Then for $i\neq j$ we have $\bp (i,j\in
  S)=\binom{M-2}{m-2}/\binom{M}{m}$, and thus
  \[
  \be \sum_{%
    \begin{subarray}{c}
      i,j \in S\\[1pt]
      i\neq j
    \end{subarray}
  } \bigabs{\bigip{g^{(n)}_i}{g^{(n)}_j}}^2 \leq M\cdot
  \binom{M-2}{m-2}\cdot \binom{M}{m}^{-1}=\frac{m(m-1)}{M-1}\ .
  \]
  Thus for some set $S$ the above inequality holds. After relabelling,
  we may assume that $S=\{1,2,\dots,m\}$. Let us now define
  $B\colon\ell_2^m\to\ell_2^{u_n}$ by $Be_i=g^{(n)}_i$ for $1\leq
  i\leq m$, where $(e_i)$ is the standard basis of
  $\br^m$. By~\eqref{eq:almost-on} we have $\norm{B}\leq 2$. Let
  $P\colon\ell_\infty^{v_n}\to\ell_\infty^m$ be the projection onto
  the first $m$ co-ordinates. Finally, for each $1\leq i\leq m$ set
  $u_i=\big( \ip{g^{(n)}_i}{g^{(n)}_j} \big)_{j=1}^m$. Observe that if
  $\bignorm{\sum_{i=1}^m\lambda_ie_i}_{\ell_\infty^m}=1$, then
  \begin{align*}
    \Bignorm{\sum_{i=1}^m \lambda_i (e_i-u_i)}_{\ell_\infty^m} &\leq
    \sum_{i=1}^m \norm{e_i-u_i}_{\ell_\infty^m} = \sum_{i=1}^m\max_{%
      \begin{subarray}{c}
        1\leq j\leq m\\[1pt]
        j\neq i
      \end{subarray}
    }\bigabs{\bigip{g^{(n)}_i}{g^{(n)}_j}}\\[2ex]
    &\leq m\cdot \sqrt{\frac{m(m-1)}{M-1}}<\frac12\ .
  \end{align*}
  It follows that there is a well-defined linear map $A\colon
  \ell_\infty^m\to\ell_\infty^m$ with $Au_i=e_i$ for
  all~$i$. Moreover, the above calculation shows that $\norm{A}\leq
  2$. We complete the proof by observing the factorization
  $I=APT_nB$.
\end{proof}

\begin{proof}[Proof of Theorem~\ref{thm:structure-of-ideals}]
  It follows from Theorem~\ref{thm:l-p-to-co} that if $[M]<[N]$
  for infinite subsets $M$ and $N$ of $\bn$, then $\cJ^{T_M}\subsetneq
  \cJ^{T_N}$. Next, recall that $I_{U,V}$ can be viewed
  as the diagonal
  operator whose $m^{\text{th}}$ diagonal entry is the formal identity
  $I\colon\ell_2^{u_m}\to\ell_\infty^{u_m}$. It then follows from
  Lemma~\ref{lem:formal-id-factors-through-T_n} that $I_{U,V}$
  factors through $T_M$ for every infinite subset $M$ of $\bn$. Thus,
  $\varf([M])<\varf([N])$ whenever $M$ is finite and $N$ is
  infinite. This completes the proof that $\varf$ is strictly
  monotonic.

  To show that $\varf$ preserves the join operation, it is of course
  sufficient to consider incomparable elements of $\fB$. In
  particular, it is enough to show that
  $\cJ^{T_M}\vee\cJ^{T_N}=\cJ^{T_{M\cup N}}$ for infinite sets $M$ and
  $N$. The left-to-right inclusion is clear, since both $T_M$ and
  $T_N$ trivially factor through $T_{M\cup N}$. Conversely,
  $T_{M\setminus N}$ factors through $T_M$, and so
  $T_{M\cup N}=T_{M\setminus N}+T_N\in
  \cJ^{T_M}+\cJ^{T_N}\subset\cJ^{T_M}\vee\cJ^{T_N}$. This shows the
  reverse inclusion, as required.

  It follows from the properties established so far that $\varf$ is
  injective. Finally, Corollary~\ref{cor:our-maps-are-fss} completes
  the proof of the theorem.
\end{proof}

\begin{rem}
  Some comments about the choice of $\varf([M])$, $M$ finite, are in
  order. We have $V=\co$ and, when $1<p<\infty$, we have $U\sim
  \ell_p$. However, the formal inclusion $I_{\ell_p,\co}$ is different
  from $I_{U,V}$. Since
  $I_{\ell_p,\co}$ factors through every non-compact operator in
  $\cL(\ell_p,\co)$, we have $\cJ^{I_{\ell_p,\co}}\subset
  \cJ^{I_{U,V}}$. When $2\leq p<\infty$ then $I_{U,V}$ factors through
  $I_{\ell_p,\co}$ via the formal inclusion from $U=\big(
  \bigoplus_{n=1}^\infty \ell_2^{u_n} \big)_{\ell_p}$ to $\ell_p=\big(
  \bigoplus_{n=1}^\infty \ell_p^{u_n} \big)_{\ell_p}$. It follows that
  in this case the ideals $\cJ^{I_{\ell_p,\co}}$ and $\cJ^{I_{U,V}}$
  are identical. However, for $1<p<2$, we have
  $\cJ^{I_{\ell_p,\co}}\subsetneq \cJ^{I_{U,V}}$, and hence our choice
  of $\varf([\emptyset])$ yields a stronger statement. We do
  not know whether $\cJ^{I_{U,V}}=\bigcap_{M\subset
    \bn,\ \abs{M}=\infty} \cJ^{T_M}$, or indeed whether $\varf$
  preserves the meet operation, \ie whether $\varf$ is a lattice
  isomorphism. To see that $I_{U,V}\notin \cJ^{I_{\ell_p,\co}}$ for
  $1<p<2$, define
  for each $m\in\bn$ a functional $\Psi_m$ on $\cL(U,V)$ by setting
  \[
  \Psi_m(S)=\tfrac{1}{u_m} \sum_{i=1}^{u_m}
  \bigip{Se^{(m)}_i}{e^{(m)*}_i}\qquad (S\in\cL(U,V))
  \]
  where $(e^{(m)}_i)_{i=1}^{u_m}$ is the unit
  vector basis of $\ell_2^{u_m}$ in $U=\big(\bigoplus_{n=1}^\infty
  \ell_2^{u_n}\big)_{\ell_p}$, and $(e^{(m)*}_i)_{i=1}^{u_m}$ is the
  unit vector basis of $\ell_1^{u_m}$ in
  $V^*=\big(\bigoplus_{n=1}^\infty \ell_1^{u_n} \big)_{\ell_1}$. It is
  easy to verify that $\Psi_m\in\cL(U,V)^*$ with $\norm{\Psi_m}\leq
  1$ and that $\Psi_m(I_{U,V})=1$ for all $m\in\bn$. On the other
  hand, given $A\in\cL(\co,V)$ and $B\in\cL(U,\ell_p)$ with
  $\norm{A}\leq 1$ and $\norm{B}\leq 1$, we have
  \[
  \bigabs{\Psi_m(AI_{\ell_p,\co}B)} \leq \tfrac1{u_m}
  \sum_{i=1}^{u_m} \bignorm{I_{\ell_p,\co}B_m e^{(m)}_i}
  \]
  where $B_m\colon \ell_2^{u_m}\to \ell_p$ is the restriction of $B$ to
  $\ell_2^{u_m}$. An application of~\cite{sz:15}*{Lemma 4} now
  gives $\Psi_m(AI_{\ell_p,\co}B)\to 0$ as $m\to\infty$. Thus, letting
  $\Psi$ be a weak$^*$-accumulation point of $(\Psi_m)$ in
  $\cL(U,V)^*$, we obtain $\Psi(I_{U,V})=1$ and
  $\cJ^{I_{\ell_p,\co}}\subset\ker \Psi$.
\end{rem}

\section{Closed ideals in $\cL(\ell_p,\ell_\infty)$ and
  $\cL(\ell_1,\ell_p)$}
\label{sec:secondary-results}

In this section we prove results analogous to
Theorems~\ref{thm:l-p-to-co} and~\ref{thm:structure-of-ideals}
for $\cL(\ell_p,\ell_\infty)$ and $\cL(\ell_1,\ell_p)$ in the range
$1<p<\infty$. In both cases a fairly straightforward modification of
the proof of Theorem~\ref{thm:l-p-to-co} will suffice. We shall also
present a duality argument that establishes some, but not all, of the
results for $\cL(\ell_1,\ell_p)$ directly from the results for
$\cL(\ell_p,\ell_\infty)$. Such a duality argument was used by
Sirotkin and Wallis in~\cite{sirotkin-wallis:16}, and we shall say
more about that in the next section.

Fix $1\leq p<\infty$ and set $U=\big(\bigoplus_{n=1}^\infty
\ell_2^{u_n} \big)_{\ell_p}$, $V=\big(\bigoplus _{n=1}^\infty
\ell_\infty^{v_n} \big)_{\co}$ and $W=\big(\bigoplus _{n=1}^\infty
\ell_\infty^{v_n} \big)_{\ell_\infty}$, where $(u_n)$ and $(v_n)$ are
sequences satisfying~\eqref{eq:defn-of-s_n}--\eqref{eq:besselian} on
page~\pageref{eq:besselian}. Similar to the previous
section, we will prove our first result for $\cL(U,W)$, which of
course is the same as working with $\cL(\ell_p,\ell_\infty)$ when
$1<p<\infty$. We shall also use the operators $T_M\colon U\to V$
defined in the previous section. Finally, we will denote by $J$ the
formal inclusion $I_{V,W}$ of $V$ into $W$.

\begin{thm}
  \label{thm:l-p-to-l_infty}
  Let $M$ and $N$ be infinite subsets of $\bn$.
  \begin{mylist}{(ii)}
  \item
    If $M\setminus N$ is infinite, then $J\circ T_M\notin
    \cJ^{T_N}(U,W)$.
  \item
    If $N\setminus M$ is finite, then $\cJ^{T_N}(U,W)\subset
    \cJ^{T_M}(U,W)$.
  \end{mylist}
  It follows that the map $\varf\colon\fB\to\fH(U,W)$ given
  by
  \[
  \varf([M])=%
  \begin{cases}
    \cJ^{T_M}(U,W) & \text{if $M$ is infinite, and}\\
    \cJ^{I_{U,V}}(U,W) & \text{if $M$ is finite,}
  \end{cases}
  \]
  is well defined, injective, monotone and preserves the join
  operation. In particular, $\cL(U,W)$ has continuum many closed
  ideals between $\cJ^{I_{U,V}}$ and $\fss$.
\end{thm}
\begin{proof}
  We follow the proof of Theorem~\ref{thm:l-p-to-co}. For $m\in\bn$
  and $S\in\cL(U,W)$ we define
  \[
  \Psi_m(S)=\tfrac{1}{v_m} \sum_{i=1}^{v_m}
  \bigip{Sg^{(m)}_i}{f^{(m)}_i}
  \]
  where $f^{(m)}_i\in W^*$ are chosen so that
  $\norm{f^{(m)}_i}=1$ and $J^*f^{(m)}_i=e^{(m)}_i$ for each~$i$. (In
  other words, $f^{(m)}_i$ is a Hahn--Banach extension of $e^{(m)}_i$
  to $W$.) We
  then have $\Psi_m\in\cL(U,W)^*$ with $\norm{\Psi_m}=1$ for
  all $m\in\bn$, and $\Psi_m(J\circ T_M)=\Phi_m(T_M)=1$ for all
  $m\in M$. The proof of~(i) will then be complete if we show that
  \begin{equation}
    \label{eq:kernel-of-psi}
    \lim_{m\to\infty,\ m\in M\setminus N} \Psi_m(AT_NB)=0 \qquad
    \text{for all }A\in\cL(V,W),\ B\in\cL(U)\ .
  \end{equation}
  To see this, fix operators $A\in\cL(V,W)$ and $B\in\cL(U)$, and
  assume without loss of generality that $\norm{A}\leq 1$ and
  $\norm{B}\leq 1$. We then have
  \[
  \bigabs{\Psi_m(AT_NB)} \leq \tfrac1{v_m} \sum_{i=1}^{v_m}
  \bignorm{T_NBg^{(m)}_i}\ .
  \]
  This is the analogue of inequality~\eqref{eq:ub-on-phi-m} from
  Theorem~\ref{thm:l-p-to-co}. Since the right-hand sides of these
  inequalities are identical, from here onwards we can simply follow
  the proof of Theorem~\ref{thm:l-p-to-co}. This then completes the
  proof of part~(i).

  Part~(ii) is again trivial, whereas the rest of the theorem is
  proved exactly as Theorem~\ref{thm:structure-of-ideals}.
\end{proof}

\begin{rem}
  The same argument works for any Banach space $W$ that contains a
  copy of $\co$ with $J$ being an isomorphic embedding of $V$ into
  $W$. We could also replace the ideals $\cJ^{T_M}$ with $\cJ^{J\circ
    T_M}$, and $\cJ^{I_{U,V}}$ with $\cJ^{I_{U,W}}$.
\end{rem}

We now turn to the duality argument. Given Banach spaces $X$ and $Y$,
for a subset $\cJ$ of $\cL(Y^*,X^*)$ we let $\cJ_*=\{
T\in\cL(X,Y):\,T^*\in\cJ\}$. The following result is straightforward
to verify.
\begin{prop}
  \label{prop:ideal-duality}
  Let $X$ and $Y$ be Banach spaces. The map $\cJ\mapsto \cJ_*$ is a
  monotone map from $\fH(Y^*,X^*)$ to $\fH(X,Y)$ that preserves the
  meet operation. Moreover, if $\cI\setminus\cJ$ contains a dual
  operator for some closed ideals $\cI$ and $\cJ$ in $\cL(Y^*,X^*)$,
  then $\cI_*\not\subset\cJ_*$.
\end{prop}

Note that $U$ is the dual space of $U_*$, where $U_*=\big(
\bigoplus_{n=1}^\infty \ell_2^{u_n}\big)_{\co}$ when $p=1$, and $U_*=\big(
\bigoplus_{n=1}^\infty \ell_2^{u_n}\big)_{\ell_q}$ when $1<p<\infty$
and $q$ is the conjugate index to $p$. In the next result we shall
write $W_*$ for $\big( \bigoplus_{n=1}^\infty
\ell_1^{v_n}\big)_{\ell_1}\cong\ell_1$, being the predual of $W$. Of
course, we have $W_*\cong V^*$, and the formal inclusion $J=I_{V,W}$
is nothing else but the canonical embedding of $V$ into its bidual
$V^{**}\cong W$.

\begin{thm}
  \label{thm:l-1-to-l_p}
  Let $\varf$ be the map defined in
  Theorem~\ref{thm:l-p-to-l_infty}. Define
  $\psi\colon\fB\to\fH(W_*,U_*)$ by $\psi([M])=\varf([M])_*$ for every
  $M\subset\bn$. Then $\psi$ is injective and monotone. In particular,
  $\cL(W_*,U_*)$ contains continuum many closed ideals.
\end{thm}
\begin{proof}
  Fix an infinite set $M\subset\bn$. Note that $J\circ T_M$ is the
  diagonal operator $\diag(S_n)_{n\in\bn}\colon U\to W$ where
  $S_n=T_n$ when $n\in M$, and $S_n=0$ otherwise. Define $S_M$ to be
  the diagonal operator $\diag(S_n^*)_{n\in\bn}\colon W_*\to U_*$. It
  is then clear that $S_M^*=J\circ T_M$. In particular, $J\circ T_M$
  is a dual operator. The theorem now follows from
  Theorem~\ref{thm:l-p-to-l_infty} and
  Proposition~\ref{prop:ideal-duality}.
\end{proof}

We conclude this section by proving a slightly stronger result for
$\cL(W_*,U_*)$ (and hence for $\cL(\ell_1,\ell_q)$ for $1<q<\infty$)
by following the proof of Theorem~\ref{thm:l-p-to-co}. In this context
it will be more appropriate to write $V^*$ in place of $W_*$. We shall
use the notation established in the proof of
Theorem~\ref{thm:l-1-to-l_p} and denote by $S_M\colon V^*\to U_*$ the
diagonal operator with $S_M^*=J\circ T_M$ for an infinite subset
$M$ of $\bn$.

\begin{thm}
  \label{thm:l-1-to-l_p-again}
  Let $M$ and $N$ be infinite subsets of $\bn$.
  \begin{mylist}{(ii)}
  \item
    If $M\setminus N$ is infinite, then $S_M\notin
    \cJ^{S_N}(V^*,U_*)$.
  \item
    If $N\setminus M$ is finite, then $\cJ^{S_N}(V^*,U_*)\subset
    \cJ^{S_M}(V^*,U_*)$.
  \end{mylist}
  It follows that the map $\psi\colon\fB\to\fH(V^*,U_*)$ given
  by
  \[
  \psi([M])=%
  \begin{cases}
    \cJ^{S_M}(V^*,U_*) & \text{if $M$ is infinite, and}\\
    \cJ^{I_{V^*,U_*}}(V^*,U_*) & \text{if $M$ is finite,}
  \end{cases}
  \]
  is well defined, injective, monotone and preserves the join
  operation. In particular, $\cL(V^*,U_*)$ has continuum many closed
  ideals between $\cJ^{I_{V^*,U_*}}$ and $\fss$.
\end{thm}
\begin{proof}
  For $m\in\bn$ consider $\Phi_m\in\cL(U,V)^*$ as defined at the start
  of the proof of Theorem~\ref{thm:l-p-to-co}. We also let
  \[
  \Psi_m(S)=\tfrac{1}{v_m} \sum_{i=1}^{v_m}
  \bigip{Se^{(m)}_i}{g^{(m)}_i}\qquad \text{for }S\in\cL(V^*,U_*)\ ,
  \]
  where $(e^{(m)}_i)_{i=1}^{v_m}$ is the unit vector basis of
  $\ell_1^{v_m}$ as in the proof of Theorem~\ref{thm:l-p-to-co}. It is
  clear that $\Psi_m\in\cL(V^*,U_*)^*$ with $\norm{\Psi_m}=1$. Since
  \[
  \bigip{S_Me^{(m)}_i}{g^{(m)}_i} =\bigip{e^{(m)}_i}{S_M^*g^{(m)}_i}
  =\bigip{e^{(m)}_i}{J\circ T_M g^{(m)}_i} =\bigip{T_M
    g^{(m)}_i}{e^{(m)}_i}\ ,
  \]
  it follows that $\Psi_m(S_M)=\Phi_m(T_M)=1$ for all $m\in M$. As in
  the proof of Theorem~\ref{thm:l-p-to-co} we will now show that
  \[
  \lim_{m\to\infty,\ m\in M\setminus N} \Psi_m(AS_NB)=0 \qquad
  \text{for all }A\in\cL(U_*),\ B\in\cL(V^*)\ .
  \]
  This will then complete the proof of part~(i) of
  Theorem~\ref{thm:l-1-to-l_p-again}. To see the above, we follow
  closely the proof of Theorem~\ref{thm:l-p-to-co}. We fix $m\in
  M\setminus N$ and assume that $\norm{A}\leq 1$ and $\norm{B}\leq
  1$. We then observe the following analogue
  of~\eqref{eq:ub-on-phi-m}.
  \[
  \bigabs{\Psi_m(AS_NB)}\leq \tfrac{1}{v_m} \sum_{i=1}^{v_m}
  \bignorm{S_N^*A^*g^{(m)}_i} = \tfrac{1}{v_m} \sum_{i=1}^{v_m}
  \bignorm{T_NA^*g^{(m)}_i}\ .
  \]
  Here $A^*$ is an element of $\cL(U)$, and hence this upper bound on
  $\abs{\Psi_m(AS_NB)}$ is of the same form as the right-hand side
  of~\eqref{eq:ub-on-phi-m}. Thus, the rest of the proof of
  Theorem~\ref{thm:l-p-to-co} can now be used to complete the proof
  of~(i).

  The rest of Theorem~\ref{thm:l-1-to-l_p-again} is proved exactly as
  Theorem~\ref{thm:structure-of-ideals} using the following
  observations. Firstly, dualizing
  Lemma~\ref{lem:formal-id-factors-through-T_n} shows that
  $I_{V^*,U_*}$ factors through $S_N$ for every infinite set
  $N\subset\bn$. Secondly, if $1<q<\infty$, then $U_*\sim \ell_q$,
  and if $q=\infty$ (or indeed, if $2\leq q\leq\infty$), then $S_N$
  factors through $\ell_2\cong \big(\bigoplus_{n=1}^\infty
  \ell_2^{u_n} \big)_{\ell_2}$ via $S_N$ viewed as a map to
  $V^*\to\ell_2$ followed by the formal inclusion of $\ell_2$ into
  $U_*$. In either case, we deduce that $S_N$ is finitely strictly
  singular from Theorem~\ref{thm:all-ops-l_1-to-l_q-are-fss} below.
\end{proof}

\begin{thm}
  \label{thm:all-ops-l_1-to-l_q-are-fss}
  Let $1<q<\infty$. Then  every operator $\ell_1\to\ell_q$ is finitely
  strictly singular. More generally, given Banach spaces $X$ and $Y$,
  if $X$ does not contain uniformly complemented and uniformly
  isomorphic copies of $\ell_2^n$, $n\in\bn$, and $Y$ has non-trivial
  type, then $\cL(X,Y)=\fss(X,Y)$.
\end{thm}

\begin{proof}
  We begin by recalling some results from the local theory of Banach
  spaces. A Banach space $Z$ is \emph{locally $\pi$-euclidean }if
  there is a constant $\lambda$ and a function $k\mapsto N(k)$ such
  that every subspace of $Z$ of dimension $N(k)$ contains a further
  subspace $2$-isomorphic to $\ell_2^k$ and $\lambda$-complemented in
  $Z$. This notion was introduced by Pe{\l}czy{\'n}ski and Rosenthal
  in~\cite{pelczynski-rosenthal:75}.

  Having non-trival type is equivalent to not containing
  $\ell_1^n$ uniformly (a result of Maurey and
  Pisier~\cite{maurey-pisier:76}), which in turn is equivalent to
  $K$-convexity (due to Pisier~\cite{pisier:82}*{Theorem~2.1}). Figiel
  and Tomczak-Jaegermann proved that a $K$-convex Banach space is
  locally $\pi$-euclidean~\cite{figiel-tomczak:79}. Thus in
  particular the space $Y$ in our theorem is locally $\pi$-euclidean.

  Let us now assume that $T\in\cL(X,Y)$ is not finitely strictly
  singular. Then there is an $\vare>0$ and a sequence of
  finite-dimensional subspaces $E_n$ of $X$ such that $\norm{Tx}\geq
  \vare\norm{x}$ for all $n\in\bn$ and for all $x\in E_n$, and such
  that $\dim E_n\to\infty$. Since $Y$ is locally $\pi$-euclidean, we
  may assume after passing to a subsequence and replacing the $E_n$ by
  suitable subspaces that $T(E_n)$ is $2$-isomorphic to $\ell_2^n$ and
  $\lambda$-complemented in $Y$ for some constant $\lambda$. It
  follows that $E_n$ is $C$-isomorphic to $\ell_2^n$ and
  $C$-complemented in $X$, where
  $C=\max\{2,\lambda\}\cdot\norm{T}/\vare$. This contradicts the
  assumption on $X$.
\end{proof}

\section{Further results, remarks and open problems}
\label{sec:remarks}

The main result of this section, Theorem~\ref{thm:large-ideals} below,
is concerned with the structure of ``large'' ideals in
$\cL(\ell_p,\co)$ for $1<p<\infty$. By ``large'' ideal we mean ideals
containing operators that are not finitely strictly singular. To place
Theorem~\ref{thm:large-ideals} into context, we shall begin with a
brief sketch of the proof of Sirotkin and
Wallis~\cite{sirotkin-wallis:16} that
$\cL(\ell_1,\ell_q)$, $1<q\leq\infty$, $\cL(\ell_1,\co)$, and
$\cL(\ell_p,\ell_\infty)$, $1\leq p<\infty$, contain uncountable many
closed ideals. We shall not define all the terms used and refer the
reader to~\cite{sirotkin-wallis:16} and to the work of Beanland and
Freeman~\cite{beanland-freeman:11} whose results play an important
r\^ole.

Fix $1\leq p<\infty$ and let $q$ be the conjugate index of $p$. For an
ordinal $\alpha$, $0<\alpha<\omega_1$, let $T_\alpha$ be the
$p$-convexified Tsirelson space of order $\alpha$. Recall that
$T_\alpha$ is a reflexive sequence space not containing any copy of
$\ell_r$, $1\leq r<\infty$, or $\co$. Its unit vector basis $(t_i)$ is
dominated by the unit vector basis $(e_i)$ of $\ell_p$, and moreover,
$(t_i)_{i\in F}$ uniformly dominates $(e_i)_{i\in F}$ for every
Schreier-$\alpha$ set $F$.

Let us denote by $J_\alpha$ the formal embedding of $\ell_p$ into
$T_\alpha$, which is the dual operator of the formal
embedding $I_\alpha$ of $T_\alpha^*$ into $\ell_q$ when $1<p<\infty$,
and into $\co$ when $p=1$. Fix a quotient map $Q_\alpha\colon\ell_1\to
T_\alpha^*$ and set $S_\alpha=I_\alpha\circ Q_\alpha$. Note that
$S_\alpha^*=Q_\alpha^*\circ J_\alpha$. Since $Q_\alpha^*$ is an
isomorphic embedding, it follows that $S_\alpha^*$ is an isomorphism
on $(e_i)_{i\in F}$ for every Schreier-$\alpha$ set $F$.
We shall now consider the closed ideals
$\cJ_\alpha=\cJ^{S_\alpha^*}(\ell_p,\ell_\infty)$.

Beanland and Freeman~\cite{beanland-freeman:11} introduced an
ordinal-index which for a strictly singular operator $T$ from $\ell_p$
into an arbitrary Banach space quantifies the failure of $T$ to
preserve a copy of $\ell_p$. Using methods of Descriptive Set Theory,
they proved that for such an operator $T$, there is a countable
ordinal $\alpha$ such that $T$ fails to preserve $\ell_p$ specifically
on Schreier-$\alpha$ sets. This means that for every bounded sequence
$(x_n)$ in $\ell_p$ and for any $\vare>0$ there is a Schreier-$\alpha$
set $F$ and scalars $(a_i)_{i\in F}$ such that
\[
\Bignorm{\sum_{i\in F}a_i Tx_i} <\vare \Bignorm{\sum_{i\in F} a_i
  x_i}\ .
\]
Let us temporarily say that such an operator $T$ is
\emph{$\alpha$-singular}. Although operators with this property do not
form an ideal, every operator in the closed ideal generated by them is
$\alpha\omega$-singular.

Let us now return to the closed ideals $\cJ_\alpha$ of
$\cL(\ell_p,\ell_\infty)$. Since $T_\alpha$ contains no copy of
$\ell_p$, it follows from the result of Beanland--Freeman that
$S_\alpha^*$ is $\beta$-singular for some countable ordinal
$\beta$. Set $\gamma=\beta\omega$. Since $S_\gamma^*$ is an
isomorphism on $(e_i)_{i\in F}$ for every Schreier-$\gamma$ set $F$,
it follows that $S_\gamma^*$ is not $\gamma$-singular, and hence
$S_\gamma^*\notin \cJ_\alpha$. This shows that there is an
$\omega_1$-sequence of countable ordinals $\alpha$ for which the
closed ideals $\cJ_\alpha$ are distinct, and moreover they are
distingished by dual operators. Hence by
Proposition~\ref{prop:ideal-duality}, the corresponding closed ideals
$\big(\cJ_\alpha\big)_*$ of $\cL(\ell_1,\ell_q)$ when $1<p<\infty$,
and of $\cL(\ell_1,\co)$ when $p=1$, are also distinct. We remark that
the closed ideals constructed by Sirotkin and Wallis are slightly
different, but the proof here is essentially a streamlined version of
their proof.

Let us emphasize the key aspects of the above argument. Firstly, all
operators in $\cJ_\alpha$ are $\gamma$-singular for a sufficiently
large $\gamma$. On the other hand, $J_\gamma$ is not
$\gamma$-singular, as it is isomorphic on $(e_i)_{i\in F}$ for every
Schreier-$\gamma$ set $F$. Since $Q_\gamma^*$ is an isomorphic
embedding, this property is inherited by $S_\gamma^*$ which therefore
does not belong to $\cJ_\alpha$. Since no operator in
$\cL(\ell_p,\co)$ can be uniformly isomorphic on $(e_i)_{i\in F}$
even for the class of Schreier-$1$ sets $F$, this argument cannot be
used to construct infinitely many closed ideals in $\cL(\ell_p,\co)$
for $1<p<\infty$. Moreover, one cannot use duality via
Proposition~\ref{prop:ideal-duality} either, as the $S_\alpha$ are not
dual operators. Recall that Wallis~\cite{wallis:15} observed that the
method of~\cite{sz:15} extends to yield continuum many closed ideals
in $\cL(\ell_p,\co)$ for $1<p<2$, but this approach does not seem to
go beyond this range of values for~$p$ either.
The main result of this section shows that there is
another obstruction. Theorem~\ref{thm:large-ideals} below
shows that at least for $p=2$ there are no proper closed ideals in
$\cL(\ell_p,\co)$ containing an operator that is not finitely strictly
singular. In other words, there are no proper large ideals. We begin
with a simple lemma.

\begin{lem}
  \label{lem:factoring-through-embedding}
  Let $E$ be a finite-dimensional Banach space and let $J\colon E\to
  \ell_\infty^m$ be an isomorphic embedding with $\norm{x}\leq
  \norm{Jx}$ for all $x\in E$. Then every operator $T\colon
  E\to\ell_\infty^n$ factors through $J$: more precisely, there is an
  operator  $A\colon\ell_\infty^m\to\ell_\infty^n$ such that
  $T=A\circ J$ and $\norm{A}\leq \norm{T}$.
\end{lem}
\begin{proof}
  We may assume that $\norm{T}\leq 1$. Then there are functionals
  $f_i, g_j\in E^*$ with $\norm{f_i}\leq \norm{J}$ and $\norm{g_j}\leq
  1$ such that
  \[
  Jx=\big( \ip{x}{f_i} \big)_{i=1}^m\quad\text{and}\quad Tx=\big(
  \ip{x}{g_j} \big)_{j=1}^n \quad\text{for all }x\in E\ .
  \]
  Since $\norm{x}\leq\norm{Jx}$ for all $x\in E$, it follows from the
  geometric Hahn--Banach theorem that the unit ball of $E^*$ is
  contained in $\conv\{ f_i:\, 1\leq i\leq m\}$. In particular, each
  $g_j$ can be expressed as a convex combination $g_j= \sum_{i=1}^m
  t_{i,j} f_i$. Define $A\colon \ell_\infty^m\to\ell_\infty^n$ by
  \[
  A \big((y_i)_{i=1}^m\big) = \Big(\sum_{i=1}^m t_{i,j}
  y_i\Big)_{j=1}^n\ .
  \]
  It is straightforward to verify that $\norm{A}\leq 1$ and $T=A\circ
  J$.
\end{proof}

From now on we fix $1<p<\infty$ and two diagonal operators $K$ and
$L$ defined as follows. For each $n\in\bn$ fix embeddings $K_n\colon
\ell_2^n\to\ell_\infty^{k_n}$ and $L_n\colon
\ell_p^n\to\ell_\infty^{m_n}$ satisfying
$\norm{x}\leq\norm{K_nx}\leq 2\norm{x}$ for all $x\in\ell_2^n$, and
$\norm{x}\leq\norm{L_nx}\leq 2\norm{x}$ for all $x\in\ell_p^n$, and
then set
$K=\diag(K_n)_{n\in\bn}\colon \big(\bigoplus_{n=1}^\infty \ell_2^n
\big)_{\ell_p} \to \big(\bigoplus_{n=1}^\infty \ell_\infty^{k_n}
\big)_{\co}$ and
$L=\diag(L_n)_{n\in\bn}\colon \big(\bigoplus_{n=1}^\infty \ell_p^n
\big)_{\ell_p} \to \big(\bigoplus_{n=1}^\infty \ell_\infty^{m_n}
\big)_{\co}$. Observe that by
Lemma~\ref{lem:factoring-through-embedding}, the closed ideals
$\cJ^K(X,Y)$ and $\cJ^L(X,Y)$ are independent of the particular choice
of embeddings $K_n$ and $L_n$ for any pair $(X,Y)$ of Banach spaces.

\begin{thm}
  \label{thm:large-ideals}
  Let $1<p<\infty$ and let $K,L$ be defined as above. Then
  \begin{mylist}{(ii)}
  \item
    $\cJ^L(\ell_p,\co)=\cL(\ell_p,\co)$, and
  \item
    if $T\in\cL(\ell_p,\co)\setminus\fss(\ell_p,\co)$, then
    $\cJ^K\subset\cJ^T$.
  \end{mylist}
  It follows that if $\cJ$ is a non-trivial, proper closed ideal of
  $\cL(\ell_p,\co)$, then either $\cJ=\cK$ or
  $\cJ^{I_{\ell_p,\co}}\subset\cJ\subset\fss$ or
  $\cJ^K\subset\cJ$. In particular, for $p=2$ we have the following:
  if $\cJ$ is a non-trivial, proper closed ideal of $\cL(\ell_2,\co)$,
  then either $\cJ=\cK$ or $\cJ^{I_{\ell_2,\co}}\subset\cJ\subset\fss$.
\end{thm}
\begin{proof}
  Let $T\in\cL(\ell_p,\co)$. It is well known that $T$ can be
  arbitrarily well approximated by the sum of two block-diagonal
  operators. So to show~(i), we may as well assume that
  $T=\diag(T_n)\colon\big(\bigoplus_{n=1}^\infty\ell_p^n\big)_{\ell_p}
  \to \big(\bigoplus_{n=1}^\infty\ell_\infty^{r_n}\big)_{\co}$
  for some integers $r_n$ and maps $T_n\colon
  \ell_p^n\to\ell_\infty^{r_n}$. By
  Lemma~\ref{lem:factoring-through-embedding}, there are operators
  $A_n\colon\ell_\infty^{m_n}\to\ell_\infty^{r_n}$ such that
  $\norm{A_n}\leq \norm{T_n}\leq\norm{T}$ and $T_n=A_n L_n$. Setting
  $A=\diag(A_n)_{n\in\bn}$, we obtain $T=A\circ L\in\cJ^L$, as
  required.

  Let us next consider $T\in\cL(\ell_p,\co)$ that is not finitely
  strictly singular. By standard basis arguments, after perturbing $T$
  by a compact (or even a nuclear) operator, we may assume that the
  columns of $T$ with respect to the canonical bases of $\ell_p$ and
  $\co$ are finite, and hence $T$ maps finitely supported vectors to
  finitely supported vectors. Since $T\notin\fss(\ell_p,\co)$, it
  follows that there
  exists $\vare>0$ and finite-dimensional subspaces $E_n$ of $\ell_p$
  with $\dim E_n\to \infty$ such that $\norm{Tx}\geq\vare\norm{x}$ for
  all $x\in E_n$ and for all $n\in\bn$. After perturbing the $E_n$ and
  replacing them by suitable subspaces, we may assume that for some
  $p_1<q_1<p_2<q_2<\dots$, we have $E_n$ and $T(E_n)$ are both
  contained in $\spn\{e_i:p_n\leq i\leq q_n\}$ for all
  $n\in\bn$. After passing to further subspaces of $E_n$, we may also
  assume that each $E_n$ is $2$-isomorphic to $\ell_2^n$ by
  Dvoretzky's theorem. Let us now fix for each $n\in\bn$ an
  isomorphism $J_n\colon \ell_2^n\to E_n$ satisfying
  $\frac1{\vare}\norm{x}\leq \norm{J_nx}\leq \frac2{\vare}\norm{x}$
  for all $x\in \ell_2^n$, and define $T_n$ to be the restriction of
  $T$ to $E_n$ viewed as a map $T_n\colon E_n\to \ell_\infty^{r_n}\cong
  \spn\{e_i:\,p_n\leq i\leq q_n\}$. Since
  $\norm{x}\leq \norm{T_n J_nx}$ for all $n\in\bn$ and
  $x\in\ell_2^n$, it follows
  from Lemma~\ref{lem:factoring-through-embedding} that the $K_n$ factor
  uniformly through $T_n J_n$. Hence $K$ factors through
  $\diag(T_nJ_n)\colon \big(\bigoplus_{n=1}^\infty \ell_2^n
  \big)_{\ell_p} \to \big( \bigoplus_{n=1}^\infty \ell_\infty^{r_n}
  \big)_{\co}$, and thus through $T$, as required.
 
  The rest of the theorem now follows.
\end{proof}

We conclude with some open problems. As already mentioned in the
Introduction, our method does not work for operators from $\ell_1$ to
$\co$. Sirotkin and Wallis~\cite{sirotkin-wallis:16} have shown that
$\cL(\ell_1,\co)$ contains uncountably many closed ideals, but the
following remains open.

\begin{problem}
  Does $\cL(\ell_1,\co)$ contain continuum many closed ideals?
\end{problem}

There are various other questions that arise naturally in this paper
but remain unanswered. We summarize these in the next two problems.

\begin{problem}
  Does the map $\varf$ in
  Theorem~\ref{thm:structure-of-ideals},~\ref{thm:l-p-to-l_infty}
  or~\ref{thm:l-1-to-l_p-again} preserve the meet operation?
\end{problem}
A positive answer would in particular imply that for disjoint infinite
subsets $M$ and $N$ of $\bn$ we have
$\cJ^{T_M}\cap\cJ^{T_N}=\cJ^{I_{U,V}}$, and thus the definition of
$\varf([\emptyset])$ is optimal.

\begin{problem}
  Does the inclusion $\fss\subset\cJ^K$ hold in
  Theorem~\ref{thm:large-ideals}?
\end{problem}
This would show that $\cJ\subset\fss\subset\cJ^K\subset \cJ'$ for all
``small'' ideals $\cJ$ and for all large ideals $\cJ'$, and would thus
shed further light on the lattice $\fH(\ell_p,\co)$.

Let us now turn attention to the remaining pairs of classical
sequence spaces that we have hitherto not mentioned. 

\begin{problem}
  What can be said about the lattice of closed ideals in
  $\cL(\ell_\infty,\ell_p)$, $1\leq p\leq \infty$, and in
  $\cL(\ell_\infty,\co)$? In particular, are they infinite?
\end{problem}
We note that the structure of $\fH(\co,\ell_\infty)$ on the other hand
is well understood and is as simple as possible.

\begin{thm}
  The compact operators are the only non-trivial proper closed ideal
  in $\cL(\co,\ell_\infty)$.
\end{thm}
\begin{proof}
  We need to show that if $T\colon\co\to\ell_\infty$ is a non-compact
  operator, then it generates $\cL(\co,\ell_\infty)$. By standard basis
  arguments, we find a block subspace $Y$ in $\co$ spanned by a block
  sequence $(x_n)$ of the unit vector basis $(e_n)$ such that $T$ is
  an isomorphism on $Y$. Let $Z=T(Y)$, and let $S$ be an arbitrary
  operator in $\cL(\co,\ell_\infty)$. Define $B\colon\co\to\co$ by
  $Be_n=x_n$. Since $(Tx_n)$ is equivalent to $(e_n)$, we can define
  $A_1\colon Z\to\ell_\infty$ by $A_1 (Tx_n)=Se_n$. Finally, since
  $\ell_\infty$ is injective, $A_1$ extends to an operator
  $A\colon\ell_\infty\to\ell_\infty$. Note that $S=ATB$, and hence the
  proof is complete.
\end{proof}

Finally, it follows from the work of Bourgain, Rosenthal and
Schechtman that $L_p[0,1]$ for $1<p<\infty$, $p\neq 2$, has up to
isomorphism uncountably many complemented
subspaces~\cite{bourgain-rosenthal-schechtman}. It follows easily that
$\cL(L_p[0,1])$ has uncountably many closed ideals. In fact, although
it is not known whether $L_p[0,1]$ has continuum many complemented
subspaces, it follows from~\cite{sz:15} that $\cL(L_p[0,1])$ has
continuum many closed ideals. This is because $L_p[0,1]$ contains a
complemented copy of $\ell_p\oplus \ell_2$. It is a well known
unsolved conjecture that every complemented subspace of $L_1[0,1]$
is isomorphic either to $\ell_1$ or to $L_1[0,1]$. The following
related question is therefore of interest.

\begin{problem}
  Does $\cL(L_1[0,1])$ contain infinitely many closed ideals?
\end{problem}

\end{document}